\documentclass[11pt]{amsart}
\usepackage{amssymb}
\usepackage{stmaryrd}
\usepackage[textsize=scriptsize,obeyFinal]{todonotes}
\usepackage{tikz-cd}
\usepackage{hyperref}

\newtheorem{theorem}{Theorem}[section]
\newtheorem*{theorem*}{Theorem}
\newtheorem{lemma}[theorem]{Lemma}
\newtheorem*{lemma*}{Lemma}
\newtheorem{proposition}[theorem]{Proposition}
\newtheorem*{proposition*}{Proposition}

\newtheorem*{corollary*}{Corollary}

\theoremstyle{definition}
\newtheorem{definition}[theorem]{Definition}
\newtheorem*{definition*}{Definition}

\newtheorem*{example*}{Example}

\newtheorem*{ques*}{Question}

\newtheorem*{claim*}{Claim}

\newtheorem*{remark*}{Remark}

\usepackage{bm}
\newcommand{\ii}{\item}
	
\newcommand{\inv}{^{-1}}

\newcommand{\ZZ}{\mathbb Z}

\usepackage{amsaddr}
\usepackage{mathrsfs}
\usepackage{tikz-cd}

\title[Multiplicative and exponential orthomorphisms]
{Multiplicative and Exponential Variations
of Orthomorphisms of Cyclic Groups}
\date{\today}
\author{Evan Chen}
\address{Department of Mathematics, Massachusetts Institute of Technology}
\email{evanchen@mit.edu}

\subjclass[2010]{11A07, 05A05}
\keywords{Cayley table, orthomorphism, transversal}

\newcommand{\Zc}[1]{\mathbb{Z} / #1 \mathbb{Z}}

\begin{document}

\begin{abstract}
	An orthomorphism is a permutation $\sigma$ of $\{1, \dots, n-1\}$
	for which $x + \sigma(x) \mod n$ is also a permutation on $\{1, \dots, n-1\}$.
	Eberhard, Manners, Mrazovi\'{c}, showed that the number
	of such orthomorphisms is $(\sqrt{e} + o(1)) \cdot \frac{n!^2}{n^n}$
	for odd $n$ and zero otherwise.

	In this paper we prove two analogs of these results
	where $x+\sigma(x)$ is replaced by $x \sigma(x)$
	(a ``multiplicative orthomorphism'')
	or with $x^{\sigma(x)}$ (an ``exponential orthomorphism'').
	Namely, we show that no multiplicative orthomorphisms exist
	for $n > 2$, but that exponential orthomorphisms
	exist whenever $n$ is twice
	a prime $p$ such that $p-1$ is squarefree.
	In the latter case we then estimate the number of
	exponential orthomorphisms.
\end{abstract}

\maketitle

\section{Introduction}
\subsection{Synopsis}
For us, an \emph{orthomorphism} of the cyclic group $\Zc{n}$
(for $n \ge 2$)
is a permutation $\sigma : \{1, \dots, n-1\} \to \{1, \dots, n-1\}$
such that the map $x \mapsto \sigma(x) + x$ is also
a permutation of $\{1, \dots, n-1\}$ (modulo $n$).\footnote{%
	In the literature one often takes
	$\sigma : \{0, \dots, n-1\} \to \{0, \dots, n-1\}$ instead,
	but by shifting $\sigma$ we may assume $\sigma(0)=0$,
	and so these two definitions are essentially equivalent.
	For example in \cite{Stones} the orthomorphisms
	we consider are called ``canonical'' orthomorphisms.
}
(It is possible to define an orthomorphism for a general group $G$
in exactly the same way as above, as in Evans \cite{evansgroup},
but we will not need this generality here.)

Orthomorphisms arise naturally in the study of Latin squares
(specifically pairs of ``orthogonal'' Latin squares) \cite{LatinWanless}.
They are in correspondence with several other combinatorial objects,
for example
\begin{itemize}
	\ii transversals of the addition table of $\Zc{n}$,
	\ii magic juggling sequences of period $n$,
	\ii and placements of non-attacking semi-queens on toroidal chessboards,
\end{itemize}
among others \cite{LatinWanless}.
They have thus been studied substantially.

It is a nice elementary result due to Euler \cite{euler}
that such an orthomorphism exists exactly when $n$ is odd.
In 1991, Vardi \cite{Latin33} conjectured that
for odd $n$ the number of orthomorphisms is between $c_1^n n!$ and $c_2^n n!$
for some constants $0 < c_1 < c_2 < 1$.
After some work on the upper bound \cite{Latin13,Latin17,Latin22}
and on the lower bound \cite{LatinWanless,Latin27},
Vardi's conjecture was completely resolved in 2015 when
Eberhard, Manners, and Mrazovi\'{c} proved (in our notation) the following result.
\begin{theorem*}[Eberhard, Manners, and Mrazovi\'{c}, \cite{EMM15}]
	For odd integers $n \ge 1$,
	the number of (canonical) orthomorphisms of $\Zc{n}$ is
	\[ \left( \sqrt{e} + o(1) \right) \frac{n!^2}{n^n}. \]
\end{theorem*}
In fact, the result of \cite{EMM15} holds for any abelian group
of odd order; Eberhard \cite{eberhard} extended this result to hold for
non-cyclic abelian groups of even order as well.
Variants of the problem have also been considered;
for example, \cite{Stones} considers \emph{compound orthomorphisms}
and uses them to find some congruences,
while \emph{partial orthomorphisms} are studied in \cite{Latin31}.

Our paper considers the variant of the problem in which
we replace $x + \sigma(x)$ by either $x\sigma(x)$ or $x^{\sigma(x)}$.
We lay out these definitions now.

\begin{definition}
	For $n \ge 2$, a \emph{multiplicative orthomorphism} of $\Zc{n}$
	is a permutation $\sigma : \{1, \dots, n-1\} \to \{1, \dots, n-1\}$
	for which $x \mapsto x \sigma(x)$ is also
	a permutation of $\{1, \dots, n-1\}$ (modulo $n$).
\end{definition}
\begin{definition}
	For $n \ge 2$, an \emph{exponential orthomorphism} of $\Zc{n}$
	is a permutation $\sigma : \{1, \dots, n-1\} \to \{1, \dots, n-1\}$
	for which $x \mapsto x^{\sigma(x)}$ is also
	a bijection of $\{1, \dots, n-1\}$ modulo $n$.
\end{definition}

Our main results are the following.
\begin{theorem}
	There are no multiplicative orthomorphisms of $\Zc{n}$
	except when $n = 2$.
	\label{thm:nomult}
\end{theorem}
\begin{theorem}
	There exists an exponential orthomorphism of $\Zc{n}$
	if and only if $n = 2$, $n = 3$, $n = 4$,
	or $n = 2p$, where $p$ is an odd prime such that
	\[ p - 1 = 2 q_1 q_2 \cdots q_k \]
	for distinct odd primes $q_1$, \dots, $q_k$.
	\label{thm:whichexp}
\end{theorem}

\begin{theorem}
	If $p - 1 = 2q_1 \cdots q_k$ as described in the previous theorem,
	then the number of exponential orthomorphisms is at least
	\[ \frac{(k+2)! \cdot 3^{k+1} \cdot
		2^{n-2^{k-1}}}{4(n-2)^{3 \cdot 2^{k-1}}}. \]
	\label{thm:countexp}
\end{theorem}

The rest of the paper is structured as follows.
We prove Theorem~\ref{thm:nomult} in Section~\ref{sec:nomult}.
In Section~\ref{sec:expreduce} we show that exponential
orthomorphisms only exist in the conditions described in
Theorem~\ref{thm:whichexp},
and then in Section~\ref{sec:construct} we prove Theorem~\ref{thm:countexp}
(which implies the other direction of Theorem~\ref{thm:whichexp}).

\section{No multiplicative orthomorphisms exist for $n > 2$}
\label{sec:nomult}
Throughout this section, $n \ge 2$ is a fixed integer,
and $\sigma : \{1, \dots, n-1\} \to \{1, \dots, n-1\}$
is a multiplicative orthomorphism.
Our aim is to show $n = 2$.

We first provide the following definition.
\begin{definition}
	Given $x \in \Zc{n}$,
	we define the \emph{rank} $R_n(x) = \gcd(x, n)$.
\end{definition}
We observe that $R_n(ab) \ge \max \left\{ R_n(a), R_n(b) \right\}$.
In particular, $R_n(x\sigma(x)) \ge \max\left\{ \sigma(x), x \right\}$.
However, the sequences $x$, $\sigma(x)$, $x\sigma(x)$
are supposed to be permutations of each other,
and in particular they have the same multisets of ranks.
Therefore this is only possible if
\[ R_n \left( x \sigma(x) \right) = R_n(x) = R_n(\sigma(x)) \]
for every $x$.

With this, we may begin by proving:
\begin{proposition}
	The number $n$ must be squarefree.
	\label{prop:mult_squarefree}
\end{proposition}
\begin{proof}
	Assume $q$ is a prime with $q^2 \mid n$.
	Then consider elements $x \in \Zc{n}$
	for which the exponent of $q$ in $x$ is either $0$ or $1$;
	observe that there exist $\frac{q^2-1}{q^2} n$ such $x$.
	For those elements, we necessarily have $q \nmid \sigma(x)$,
	otherwise $R_n(x\sigma(x)) \ge q R_n(x) > R_n(x)$, which is a contradiction.

	Thus at least $\frac{q^2-1}{q^2} n$ of the $\sigma(x)$'s
	need to be not divisible by $q$.
	But $\sigma$ is a permutation of $\{1, \dots, n-1 \}$,
	which only has $\frac{q-1}{q} n$ elements not divisible by $q$,
	giving a contradiction.
\end{proof}

Let $q$ now be any prime divisor of $n$,
and let $m = n/q$.
Since $n$ is squarefree we have $\gcd(m, q) = 1$.
Consider the set $S$ consisting of the $q-1$ elements of rank $m$,
namely \[ S = \{m, 2m, \dots, (q-1)m \}. \]
Then $\sigma(x)$ and $x\sigma(x)$ both induce permutations on $S$,
and therefore we have
\[  \left( \prod_{i=1}^{q-1} im  \right)^2
	\equiv \prod_{i=1}^{q-1} im \cdot \sigma(im)
	\equiv \prod_{i=1}^{q-1} im \pmod{n}. \]
As $q$ divides $n$ we conclude
$\left( \prod_{i=1}^{q-1} im \right)^2 \equiv \prod_{i=1}^{q-1} im \pmod{q}$,
Since $\gcd(im,q) = 1$ for $1 \le i \le q-1$,
we finally conclude
\[ 1 \equiv \prod_{i=1}^{q-1} im = (q-1)! \cdot m^{q-1} \pmod q. \]
By Fermat's little theorem we know $m^{q-1} \equiv 1 \pmod q$.
On the other hand, $(q-1)! \equiv -1 \pmod q$ by Wilson's theorem.
Consequently, we conclude $-1 \equiv 1 \pmod q$,
and therefore $q = 2$.

Since $q$ was any prime dividing $n$, and $n$ is squarefree,
we conclude $n = 2$ is the only possible value.

\section{Characterizing $n$ for exponential orthomorphisms}
\label{sec:expreduce}
In this section our aim is to show that
if $\sigma$ is an exponential orthomorphism modulo $n$,
then $n$ has the form described in Theorem~\ref{thm:whichexp}.

Fix $n \ge 3$ an integer
and $\sigma$ an exponential orthomorphism on $\{1, \dots, n-1\}$.

\begin{proposition}
	If $n$ is not squarefree, then $n = 4$.
\end{proposition}
\begin{proof}
	As before, we note that
	\[ R_n(x^e) \ge R_n(x) \]
	for each $x \in \Zc{n}$ and $e \in \ZZ_{>0}$.
	In particular, $R_n(x^{\sigma(x)}) \ge R_n(x)$.
	Again since $x^{\sigma(x)}$ and $x$ are permutations of each other
	we must have $R_n(x^{\sigma(x)}) = R_n(x)$ for each $x$.

	Now suppose $p$ is a prime with $p^2$ dividing $n$.
	Let $x$ be any element of $\Zc{n}$
	for which $\gcd(x, n) = p$.
	Since $R_n(x^{\sigma(x)}) > R_n(x)$ if $\sigma(x) > 1$
	we must instead have $\sigma(x) = 1$.

	In particular $\sigma(p) = \sigma(n-p) = 1$.
	This is only possible if $p = n-p$, i.e., $n = 2p$.
	Since we assumed $p^2 \mid n$, this means $p=2$ and $n=4$.
\end{proof}

Thus, we henceforth assume $n$ is a product of distinct primes.
\begin{proposition}
	If $n$ is squarefree,
	then it is either prime, or twice a prime.
 \end{proposition}
\begin{proof}
	First, suppose $n = p_1 p_2 \dots p_r$ is odd,
	where $p_1 < p_2 < \dots < p_r$ are distinct primes.
	We observe that if $r > 1$ we have
	\[ \prod_i \left( \frac{p_i+1}{2} \right) - 1 < \frac{n-1}{2}. \]
	(Indeed, we note that $\frac{p_1+1}{2} \cdot \frac{p_2+1}{2} < \frac12 p_1p_2$
	rearranges to $(p_1-1)(p_2-1) > 2$,
	and then simply use $\frac{p_i+1}{2} \le p_i$ for $i \ge 3$.)

	But the left-hand side is the number of nonzero quadratic residues
	in $\Zc{n}$ while the right-hand is the number of even elements
	in $\{1, \dots, n-1\}$.
	This is a contradiction since whenever $\sigma(x)$ is even
	the number $x^{\sigma(x)}$ is a quadratic residue,
	implying that there are at least as many quadratic residues as even numbers.

	In exactly the same way, if $n = 2 p_1 \cdots p_r$ is even
	and $r > 1$, then we obtain
	\[ 2\prod_i \left( \frac{p_i+1}{2} \right) - 1 < \frac n2 \]
	which is a contradiction in the same way.
\end{proof}

We now handle the prime case.
\begin{proposition}
	The number $n$ cannot be prime unless $n = 3$.
	\label{prop:primroot}
\end{proposition}
\begin{proof}
	Let $n$ be a prime.
	Fix an isomorphism $\theta : (\Zc{n})^\times \to \Zc{(n-1)}$
	given by taking a primitive root $g$ of $\Zc{n}$
	such that $g^{\theta(x)} \equiv x \pmod n$ for $x \in (\Zc{n})^\times$.
	This gives us a diagram
	\begin{center}
	\begin{tikzcd}
		(\Zc{n})^\times \ar[d, "\theta"] \ar[r, "\sigma"] & \{1, \dots, n-1\} \\
		\Zc{(n-1)} \ar[ru, dashed, swap, "\tilde{\sigma}"] & 
	\end{tikzcd}
	\end{center}
	where we have a natural map
	$\tilde{\sigma} : \Zc{(n-1)} \to \{1, \dots, n-1\}$
	which makes the diagram commute.

	Obviously $\sigma(1) = n-1$,
	since otherwise $1 = 1^{\sigma(1)} = (\sigma\inv(n-1))^{n-1}$.
	As $\theta(1) = 0$, we conclude $\tilde{\sigma}(0) = n-1$.
	Looking at the remaining elements,
	$\tilde{\sigma}$ induces a multiplicative orthomorphism on $\Zc{(n-1)}$,
	which we know is only possible if $n-1=2$.
	Hence we conclude $n = 3$.
\end{proof}

Thus we may henceforth assume that $n = 2p$, where $p$ is prime.
We may as well assume $p$ is odd.
Then in $\Zc{2p}$ there are three types of nonzero elements:
\begin{itemize}
	\ii The odd numbers
	$O = \{1, 3, \dots, p-1, p+1, \dots, 2p-1\}$
	(of rank $1$).
	These remain odd under exponentiation,
	and as a multiplicative group is
	isomorphic $(\Zc{2p})^\times \cong (\Zc{p})^\times \cong \Zc{p-1}$.
	\ii The even numbers
	$E = \{2, \dots, 2p-2 \}$
	(of rank $2$).
	These remain even under exponentiation,
	and as a multiplicative group is isomorphic
	$(\Zc{p})^\times$ as well.
	\ii The special element $p$ (of rank $p$),
	for which $p^c \equiv p \pmod{2p}$ for any $c \in \ZZ$.
\end{itemize}
As all the elements above have order dividing $p-1$,
we may consider the image of $\sigma$ modulo $p-1$
to obtain the multiset
\[ S = \left\{ 1,1,1,2,2,3,3,\dots,p-1,p-1 \right\} \]
of size $n-1 = 2p-1$.
In other words, we may instead consider
$\sigma : \{1, \dots, n-1\} \to S$.
Thus, for $k = 1, \dots, p-1$ viewed as elements of $(\Zc{p})^\times$,
we define
\begin{align*}
	a_k &= \begin{cases}
		\sigma(2k-1) & k \le \frac{p-1}{2} \\
		\sigma(2k+1) & k \ge \frac{p+1}{2}
	\end{cases} \\
	b_k &= \sigma(2k) \\
	c &= \sigma(p).
\end{align*}
Diagramatically,
\begin{center}
\begin{tikzcd}
	O \sqcup E \ar[r, "\sigma"] \ar[d, swap, "\simeq"] & S \\
	(\Zc{p})^\times \sqcup (\Zc{p})^\times
		\ar[ru, "(a_\bullet{,} b_\bullet)", swap] &
\end{tikzcd}
\end{center}
Thus, we have reformulated the problem as follows:
\begin{proposition}
	Assume $n = 2p$ with $p$ an odd prime.
	Then $n$ satisfies the problem conditions
	if and only if there exists a permutation
	\[ (a_1, \dots, a_{p-1}, b_1, \dots, b_{p-1}, c)
		\quad\text{of}\quad S \]
	such that
	\[ (a_1, 2a_2, \dots, (p-1)a_{p-1})
		\quad\text{and}\quad (b_1, 2b_2, \dots, (p-1)b_{p-1}) \]
	are permutations of $\Zc{(p-1)}$.
\end{proposition}

With this formulation we may now show the following.
\begin{proposition}
	If $n = 2p$ with $p$ prime, then $p-1$ is squarefree.
\end{proposition}
\begin{proof}
	This mirrors the proof of \ref{prop:mult_squarefree},
	with small modifications.
	As before we have
	\begin{align*}
		R_{p-1}\left( ka_k \right)
			&\ge \max \left\{ R_{p-1}(k), R_{p-1}(a_k) \right\}
			\ge R_{p-1}(k) \\
		R_{p-1}\left( kb_k \right)
			&\ge \max \left\{ R_{p-1}(k), R_{p-1}(b_k) \right\}
			\ge R_{p-1}(k).
	\end{align*}
	The change to the argument is that
	$a_k$ and $b_k$ are not collectively a permutation of $S$
	(since there is an extra unused element $c$).
	However, we may still conclude
	(since $ka_k$, $kb_k$ and $k$ are permutations of each other)
	that
	\[ R_{p-1}(ka_k) = R_{p-1}(kb_k) = R_{p-1}(k). \]
	Now suppose $q$ is a prime for which $q^2 \mid p-1$.
	Then as before, whenever the exponent of $q$ in $k$ is at most one,
	we would require $a_k$ and $b_k$ to not be divisible by $q$.
	So among $a_k$ and $b_k$ we need at least
	\[ 2 \cdot \frac{q^2-1}{q^2} (p-1) \]
	values to be not divisible by $q$,
	but in the multiset $S$ the number of such elements is
	\[ 1 + \frac{q-1}{q} \cdot 2(p-1)
		< 2 \cdot \frac{q^2-1}{q^2} (p-1) \]
	which is a contradiction.
\end{proof}

Together these propositions establish that $n$
must have the form described in Theorem~\ref{thm:whichexp}.

\section{Construction}
\label{sec:construct}
It remains to prove the converse of Theorem~\ref{thm:whichexp}
as well as Theorem~\ref{thm:countexp}.
This estimate requires several different components.

\subsection{Decomposition of functions as sums of two permutations}
We take the following lemma from \cite{SL2005C7}.
\begin{lemma}
	\label{lem:2005C7}
	Let $G$ be a finite abelian group.
	Given a function $f \colon G \to G$
	for which $\sum_{g \in G} f(g) = 0$,
	there exist two permutations $\pi_1, \pi_2 \colon G \to G$
	for which \[ f = \pi_1 + \pi_2. \]
\end{lemma}
The results of \cite[Theorem 1.3]{eberhard} suggest
that it may be possible to improve this bound significantly
given ``reasonable'' assumptions on $f$,
but we will not do so here.

\subsection{Splitting Lemma}
For a set $T$ let $\Sigma T$ denote the sum of the elements of $T$.
We prove the following result.
\begin{lemma}
	\label{lem:split}
	Let $G$ be a finite abelian group of order $N$,
	and let $S = G \coprod G$ be considered a set of $2N$
	distinct elements.
	Then there exist at least
	\[ \frac{4^N}{2(N+1)^{\frac32}} \]
	subsets $T \subset S$ for which $|T| = N$,
	$\Sigma T = 0$.
\end{lemma}
\begin{proof}
	According to the structure theorem of abelian groups
	we may write $G = \Zc{r_1} \times \dots \times \Zc{r_m}$,
	where $r_1 \mid r_2 \mid \dots \mid r_m$.
	In this way, we may think of each element $g \in G$
	as a vector $g = (g_1, \dots, g_m) \in G$.
	(In particular $(\Sigma T)_j$ refers to the $j$th coordinate of $\Sigma T$,
	since $\Sigma T \in G$).

	For each $i$ let $\zeta_i$ be a primitive $r_i$th root of unity,
	and let $\eta$ be a primitive $N$th root of unity.
	We now define
	\begin{align*}
		F(e_1, \dots, e_m, d)
		&= \prod_{g \in G} \left( 1 +
			\zeta_1^{e_1 g_1} \cdots \zeta_m^{e_m g_m} \eta^{d} \right)^2 \\
		&= \prod_{g \in S} \left( 1 +
			\zeta_1^{e_1 g_1} \cdots \zeta_m^{e_m g_m} \eta^{d} \right). \\
		\intertext{Expanding completely, we also have the representation}
		F(e_1, \dots, e_m, d)
		&= \sum_{T \subset S}
			\zeta_1^{e_1 (\Sigma T)_1}
			\cdots \zeta_m^{e_m (\Sigma T)_m}
			\eta^{d |T|}.
	\end{align*}
	Now consider the sum
	\[ A = \sum_{e_1=0}^{r_1-1} \dots \sum_{e_m=0}^{r_m-1}
		\sum_{d=0}^{N-1} F(e_1, \dots, e_m, d). \]
	On the one hand, we find that
	\begin{align*}
		A &= \sum_{e_1=0}^{r_1-1} \dots \sum_{e_m=0}^{r_m-1} \sum_{d=0}^{N-1}
		\left[ \sum_{T \subset S} \zeta_1^{e_1 (\Sigma T)_1}
			\cdots \zeta_m^{e_m (\Sigma T)_m} \eta^{d |T|} \right] \\
		&= \sum_{e_1=0}^{r_1-1} \dots \sum_{e_m=0}^{r_m-1}
		\left[ \sum_{T \subset S} \zeta_1^{e_1 (\Sigma T)_1}
			\cdots \zeta_m^{e_m (\Sigma T)_m}
			\left[ \sum_{d=0}^{N-1} (\eta^{|T|})^d \right] \right]. \\
		\intertext{Note that the innermost sum is $N$ if $|T| \equiv 0 \pmod n$,
		and $0$ otherwise. Thus we may now write}
		A &= \sum_{\substack{T \subset S
				\\ |T| \equiv 0 \pmod n}}
			N \prod_{i=1}^m \left(  \sum_{e_i=0}^{r_i-1}
				\zeta_i^{e_i (\Sigma T)_i} \right) \\
		&= \sum_{\substack{T \subset S
				\\ |T| \equiv 0 \pmod n \\ \Sigma T = 0}}
			N r_1 \cdots r_m \\
		&= N^2 \left\lvert \left\{ T \subset S :
			|T| \equiv 0 \pmod n, \; \Sigma T = 0 \right\} \right\rvert \\
		&= N^2 \left( 2 + \left\lvert \left\{ T \subset S :
			|T| = n, \; \Sigma T = 0 \right\} \right\rvert \right).
	\end{align*}
	On the other hand, we have the bounds
	\[ |F(e_1, \dots, e_m, d)| < \left( 2^{\frac{N}{r_i}} \right)^2
		\text{ if } e_i \neq 0. \]
		Moreover,
	\[ \sum_d F(0,\dots,0,d) = \sum_d (1+\eta^d)^{2N}
		= N \left( 2 + \binom{2N}{N} \right). \]
	Thus, we have the estimate
	\[ A \ge N \left( 2 + \binom{2N}{N} \right) - N(N-1) \cdot 2^N \]
	and consequently
	\[
		\# \left\{ T \subset S :
			|T| = n, \; \Sigma T = 0 \right\}
		\ge -2 + \frac{2 + \binom{2N}{N} - (N-1) \cdot 2^N}{N}.
	\]
	Using the estimate $\binom{2N}{N} \ge \frac{4^N}{\sqrt{4N}}$
	one can verify the above is at least
	\[ \frac{A}{N^2}-2 \ge \frac{4^N}{2(N+1)^{3/2}} \]
	for $N \ge 8$.
	All that remains is to examine the cases $N \le 7$,
	which can be checked by hand by explicitly computing $A$.
\end{proof}
\begin{remark*}
	Lemma~\ref{lem:split} has appeared in various specializations;
	for example, the case where $G = \Zc{p}$ was
	the closing problem of the 1996 International Mathematical Olympiad,
	in which the exact answer
	$\frac1p \left( \binom{2p}{p} -2 \right) + 2$ is known.
\end{remark*}

\subsection{Main construction}
We now prove Theorem~\ref{thm:countexp}.
\begin{proof}
We begin by constructing a partially ordered set
on the divisors of $p-1 = 2q_1 \cdots q_k$, ordered by divisibility;
hence we obtain the Boolean lattice with $2^{k+1}$ elements.
At the node $d$ in the poset we write down the elements $x \in \{1, \dots, n-1\}$
for which $\gcd(x,p-1) = d$;
this gives $2\varphi( (p-1)/d )$ elements written at each node except the first one,
for which we have $2\varphi(p-1)+1$ elements.

Then, we iteratively repeat the following process,
starting at the bottom node $d=1$:
\begin{itemize}
	\ii Note there are three labels which are $1 \pmod{\frac{p-1}{d}}$.
	Pick one of these three numbers $x$ arbitrarily, and erase it.
	\ii If $d=p-1$, stop.
	Otherwise, pick one node $d'$ immediately above $d$,
	and write $x$ at that node $d'$.
	\ii Move to the node $d'$,
	which now has three labels which are $1 \pmod{\frac{p-1}{d'}}$,
	and continue the process.
\end{itemize}
An example of this process with $n=14$ (giving $p-1=6$)
is shown in Figure~\ref{fig:ex1}.

\begin{figure}[ht]
	\begin{center}
	\begin{tikzcd}
		& 6: \{6, 12\} \ar[ld, dash] \ar[rd, dash] & \\
		2: \{2,4,8,10\} \ar[rd, dash] & & 3: \{3, 9\} \ar[ld, dash] \\
		& 1: \{1,5,7,11,13\} &
	\end{tikzcd}
	\bigskip
	\begin{tikzcd}
		& 6: \{6, \mathbf{10} \} \ar[ld, leftarrow, "10", swap] \ar[rd, dash]
		\ar[r, "12"]
		& (\text{delete } 12) \\
		2: \{2,4,\mathbf{7},8\} \ar[rd, leftarrow, "7"] & & 3: \{3, 9\} \ar[ld, dash] \\
		& 1: \{1,5,11,13\} &
	\end{tikzcd}
	\end{center}
	\caption{An example of the algorithm described.
	The initial poset before the algorithm is shown on top.
	Thereafter, we pick the chain $1 \to 2 \to 6$
	and move the elements $7$, $10$, $12$.
	This gives the poset at the bottom.}
	\label{fig:ex1}
\end{figure}
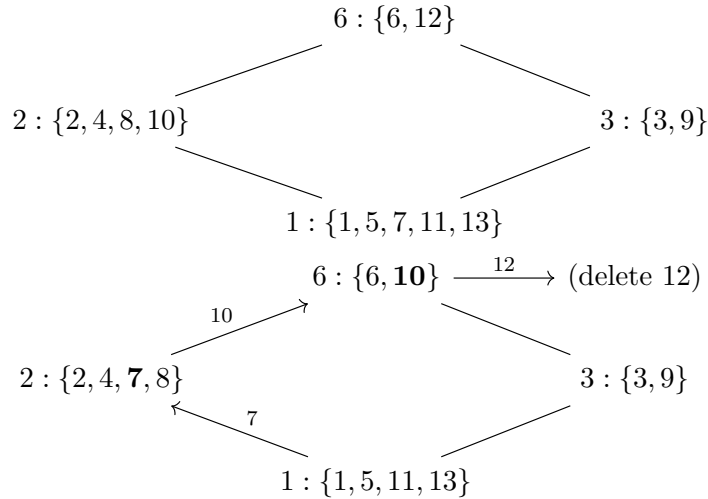

Evidently, there are $3^{k+2} (k+1)!$ ways to run the algorithm,
and each application gives a different set of labels at the end.
We will use each labeled poset to exhibit several exponential orthomorphisms.
For each $d \mid p-1$, let $L_d$ denote the labels at the node $d$.

As in the previous section,
we identify all the elements of $\{1, \dots, 2p-1\} \setminus \{p\}$
with the set
\[ Z = E \sqcup O = (\Zc{p})^\times \sqcup (\Zc{p})^\times. \]
Now consider any $d \mid p-1$, let $e= \frac{p-1}{d}$
and let $m = \varphi(e)$.
There are $2m$ elements $x \in Z$
for which $R_{p-1}(x) = d$;
they can be thought of as $G \sqcup G$ where
$G = (\Zc{\frac{p-1}{d}})^\times \cong \Zc{m}$.
The labels written at node $d$ can be thought of in the same way.

We will match these to the labels written at the node $d$ in our poset.
By Lemma~\ref{lem:split}, the number of ways to split the labels 
into two halves $L = L_E \sqcup L_O$,
such that each half has vanishing product,
is at least
\[ \max\left(\frac{4^{m}}{2(m+1)^{3/2}}, 2\right)
	\ge \frac{4^{\varphi(e)}}{2e^{3/2}}. \]
(Here we have used the fact that $\varphi(e)+1 \le e$ for $e \neq 1$).
Moreover, by Lemma~\ref{lem:2005C7},
there exists at least one way to choose a bijection
$\sigma \colon E \to L_E$ so that the map $x \mapsto x\sigma(x)$
is a bijection on $E$;
of course the analogous result holds for $\sigma \colon O \to L_O$.
Hence we've defined $\sigma$ as a bijection
on the elements $x \in Z$ with $R_{p-1}(x) = d$, as desired.

Finally, we label the special element $p$
with the single unused number left over from the algorithm.
Thus we get a bijection $\sigma$ on
the entirety of $\{1, \dots, 2p-1\}$.

The number of orthomorphisms we've constructed is at least
\begin{align*}
	(k+2)! \cdot 3^{k+1}
	\prod_{e \mid p-1} \frac{4^{\varphi(e)}}{2e^{3/2}}
	&= (k+2)! \cdot 3^{k+1}
	\frac{4^{p-1}}{2^{2^{k+1}} \left[ (p-1)^{2^k} \right]^{3/2}} \\
	&= (k+2)! \cdot 3^{k+1}
	\frac{2^{n-2}}{2^{2^{k+1}} \left( \frac{n-2}{2} \right)^{3 \cdot 2^{k-1}}} \\
	&= (k+2)! \cdot 3^{k+1}
	\frac{2^{n-2-2^{k+1}+3\cdot2^{k-1}}}{(n-2)^{3 \cdot 2^{k-1}}} \\
	&= \frac{(k+2)! \cdot 3^{k+1} \cdot
		2^{n-2^{k-1}}}{4(n-2)^{3 \cdot 2^{k-1}}}.
\end{align*}
This concludes the proof.
\end{proof}

\subsection*{Acknowledgments}
This research was funded by NSF grant 1659047,
as part of the 2017 Duluth Research Experience for Undergraduates (REU).
The author thanks Joe Gallian for supervising the research,
and for suggesting the problem.
The author is also grateful to Joe Gallian and Ian Wanless,
as well as the anonymous referee, for comments on drafts of the paper.

\bibliographystyle{hplain}
\bibliography{refs}

\end{document}